\documentclass[]{amsart}
\usepackage{amsmath,amsfonts,amssymb,amsthm}
\usepackage{enumerate}
\usepackage[usenames]{color}
\usepackage{graphicx}
\usepackage[abs]{overpic}

\DeclareMathOperator{\Sing}{Sing}

 \newtheorem{theorem-main}{\bf Theorem}
\newtheorem{theorem}{\bf Theorem}
 \newtheorem{lemma}[theorem]{\bf Lemma}
 \newtheorem{proposition}[theorem]{\bf Proposition}

  \newtheorem{rk}[theorem]{\bf Remark}
  \newtheorem{example}[theorem]{\bf Example}
  
    \newtheorem*{claim}{\bf Claim}
      \newtheorem*{assertion}{\bf Assertion}

 \newcommand{\R}{\mathbb{R}}
 \newcommand{\Q}{\mathbb{Q}}

  \newcommand{\C}{\mathbb{C}}
    \newcommand{\D}{\mathbb{D}}
 
 \newcommand{\SSS}{\mathbb{S}}

\newcommand{\eps}{\varepsilon}

\newcommand{\wt}[1]{{{\widetilde{#1}}}}

\def\R{\mathbb R}
\def\C{\mathbb C}
\def\Q{\mathbb Q}

\newcommand{\g}{\gamma}
\newcommand{\G}{\Gamma}

\renewcommand{\wt}[1]{\widetilde{#1}}

 \renewenvironment{proof}{{\em \noindent Proof.}}
 {\hfill $\square$\newline}

 \newcommand{\obra}[3]{{\sc #1} {\em #2}. {#3}.}

\author{Rog\'erio Mol}
\address{Departamento de Matem\'{a}tica\\
Universidade Federal de Minas Gerais\\
Av. Ant\^onio Carlos, 6627, CP 702, 30123-970
Belo Horizonte, Brazil}
\email{rmol@ufmg.br}
\thanks{First author partially supported by Pronex-Faperj, CNPq-Universal, CAPES-Mathamsud}

\author{Fernando Sanz S\'{a}nchez}
\address{Departamento de \'{A}lgebra, An\'{a}lisis Matem\'{a}tico, Geometr\'{\i}a y Topolog\'{\i}a\\
Universidad de Valladolid\\
Paseo de Bel\'en 7, 47011 Valladolid, Spain}
\email{fsanz@agt.uva.es}
\thanks{Second author
partially supported
by Ministerio de Educaci\'on y Cultura, Spain, process MTM2013-46337-C2-1-P and MTM2016-77642-C2-1-P, and by Programa Hispano-Brasile\~no de Cooperaci\'on Interuniversitaria, process PHB2010-0122-PC}

\title[Vector Fields and Separatrices]
{Real Analytic Vector Fields with First Integral and Separatrices}
\date{}
\dedicatory{To our teacher  Felipe Cano, with immense gratitude}

\begin{document}

\begin{abstract} We prove that   a germ      of  analytic vector field at $(\mathbb{R}^3,0)$
  that possesses a non-constant analytic first integral
 has a real formal separatrix. We provide an example which shows that such a vector field does not  necessarily have a real analytic separatrix.
\end{abstract}
\maketitle

\section{Introduction}
In this paper we prove the following result:

\begin{theorem}
\label{th:main}
Let $X$ be a germ of real analytic vector field at $(\mathbb{R}^3,0)$
 that has an analytic first integral. Then $X$ has a
    real formal separatrix.  The  statement is optimal in the sense that such a vector field $X$ does not necessarily have a real analytic separatrix.
\end{theorem}

 Speaking in general terms, let $X$ be a germ of      real analytic vector field at $(\mathbb{R}^n,0)$.  A
{\em real analytic separatrix} of $X$ is a germ of  irreducible analytic
curve $\G$ at $0 \in \mathbb{R}^n$ which is invariant by $X$. If
$\g(t)=(\g_1(t),...,\g_n(t))\in (t\R\{t\})^n\setminus\{0\}$ is a parametrization of $\G$, the
invariance condition is equivalent to saying  that there exists $h(t) \in \R\{t\}$  such that $X(\g(t))=h(t) \frac{d\g}{dt}(t)$ for any $t$, where
$h(t) \not\equiv 0$ if and only if $\G$ is
not contained in the singular
locus $\Sing(X)=\{p\, ;X(p)=0\}$ of $X$.
  Replacing $\R\{t\}$ by
 $\R[[t]]$,  we obtain
  the concept of
{\em real formal separatrix}.
 On the other hand,
considering the canonical complexification of $X$ to a holomorphic vector field at $(\mathbb{C}^n,0)$  and changing $\R$ to $\C$, we have the  concepts of {\em complex holomorphic se\-pa\-ra\-trix} and
 {\em complex formal separatrix},  seen as  objects in $(t\C\{t\})^{n}\setminus\{0\}$ and  $(t\C[[t]])^{n}\setminus\{0\}$, respectively.

We also recall that a first integral of $X$ is a germ of function $f:(\R^n,0)\to\R$ such that $df(X)=0$. The expression ``analytic first integral'' in Theorem~\ref{th:main} could be interpreted either as ``holomorphic first integral'' or ``real analytic first integral''. In fact, if $h:(\C^3,0)\to\C$ is a non-constant holomorphic first integral of (the complexification of) $X$, then one can check that the real traces of $Re(h)$ and $Im(h)$ are real analytic first integrals of $X$  with at least one of them non-constant.

Notice that  in Theorem~\ref{th:main} we may assume without loss of generality that $X$ has an isolated singularity at $0$, otherwise there is at least a real analytic separatrix of $X$ contained in $\Sing(X)$. On the contrary, we do not assume necessarily that the singular locus $\Sing(df)= \{p; df(p) = 0\}$ of the  first integral $f$ of $X$ is isolated. However, taking into account that   $\Sing(df)$   is invariant by the vector field $X$, we may assume that it has no one-dimensional real components (see below in Section~\ref{section-proof} for details).

Analytic or formal separatrices may of course be defined  for  holomorphic vector fields. They are algebraically manipulable invariant objects which play a central role in the study of the local dynamics of the vector field.  Let us briefly review some  avatars of the problem of existence of separatrices, related to the situation of real vector fields.

\strut

{\em Planar case, $n=2$.} First,  the Separatrix
Theorem of Camacho and Sad  \cite{Cam-S} asserts that a planar vector field always has  a complex
holomorphic separatrix, although it may not have formal real
separatrices: take, for instance, the standard vector field of {\em
center-type},
$
X=-y\frac{\partial}{\partial x}+x\frac{\partial}{\partial y}.
$
In this example, $X$ has an analytic first integral, showing that
Theorem~\ref{th:main} is not true for planar vector fields.
On the other hand, there are
examples of planar real vector fields with real formal separatrices, none
of them  convergent. An explicit example could be found in \cite[Example 3.7(3)]{Ris}.   Below, in Section~\ref{sec:example}, we provide other examples used for the proof of the second part of Theorem~\ref{th:main}.

It is also known that an analytic vector field $X$ at $0\in\R^2$ with Poincar\'{e} index equal to zero has a real formal separatrix.
Below, in Proposition~\ref{pro:index-separatrix}, we provide a ge\-ne\-ra\-li\-za\-tion of this result for vector fields defined in singular analytic surfaces,
 which is one of the main ingredients of the  proof of Theorem~\ref{th:main}.


\strut

{\em Three dimensional case, $n=3$.} Camacho-Sad's Theorem is no longer valid   in this case:
G\'{o}mez-Mont and Luengo in \cite{Gom-L} have constructed a family of
 vector fields in $(\C^3,0)$
without complex separatrices. They state the result for analytic separatrices, although the same proof works
in order to show that any vector field in that family is actually devoid of complex formal separatrices.
An explicit member of that family with real coefficients could be found in \cite[p. 333]{Oli}. As a consequence of Theorem~\ref{th:main},  vector fields in G\'omez-Mont and Luengo's family with real coefficients cannot have non-constant holomorphic first integrals.
%

As for the planar case, there are examples of analytic vector fields at $(\R^3,0)$
with formal real
separatrices, none of them convergent (i.e. without real analytic separatrices). An explicit example can be found in \cite[p. 3]{Can-M-S2}. We construct in Section~\ref{sec:example} another example  which has, moreover, a non-constant analytic first integral.  It   will prove the second part of Theorem~\ref{th:main}, that is, that the conclusion ``formal'' in the statement cannot be improved to ``analytic''.

We should mention that, in a recent paper,
D. Cerveau and A. Lins Neto proved that a germ of complex  analytic vector field in $(\mathbb{C}^{3},0)$,
with isolated singularity, that is tangent to a holomorphic foliation of codimension one always has
a complex analytic separatrix \cite[Proposition 3]{Ce-LN}.
This result  implies in particular that  any vector field
$X$ as  in Theorem 1 actually {\em has} a complex analytic separatrix, inasmuch as $X$ is tangent to the foliation $df=0$, where $f$ is a first integral.  Such a complex separatrix may not be a real one (once more by our example in Section~\ref{sec:example} below).

\strut

{\em Higher dimension, $n\geq 4$.} Families of holomorphic vector fields at $(\C^{n},0)$ without complex separatrices (neither convergent nor formal)   are constructed in \cite{Lue-O} for any dimension $n\geq 4$, generalizing the three dimensional construction carried out in \cite{Gom-L}. Each one of these families contains an explicit example with real coefficients.

Examples of real analytic vector fields without   real formal separatrices having   analytic first integral can be constructed in any dimension $n \geq 4$, showing that the phenomenon depicted in
 Theorem~\ref{th:main} is   exclusive for   dimension  three.
When $n=2p$ is even, we   consider a {\em multicenter} vector field, written in coordinates $(x_1,y_1,...,x_p,y_p)$ as $Z_{n} =X_1+\cdots+X_p$, where $X_j=-y_j\frac{\partial}{\partial x_j}+x_j\frac{\partial}{\partial y_j}$. When $n = 2p+3$ is odd, $p \geq 1$, we take coordinates $(x_1,y_1,...,x_p,y_p, x,y,z)$ and set $Z_{n} = Z_{2p} + W$,
where $Z_{2p}$ is  a multicenter  vector field in  the variables $(x_1,y_1,...,x_p,y_p)$  and  $W$ is  one of the examples of three dimensional real vector fields in G\'{o}mez-Mont and Luengo's family  written in the variables $(x,y,z)$.
 Notice that, in both cases, $Z_n$  has  $f(x_1,y_1) = x_{1}^{2} +  y_{1}^{2}$ as a first integral.

Finally, concerning real analytic separatrices in any dimension, it is worth  mentioning  Moussu's paper \cite{Mou}, where it is proved that an  analytic gradient vector field at $(\R^{n},0)$
always has a real analytic separatrix.  Below, we describe some arguments of that result, those which are used in our proof of Theorem~\ref{th:main} (concretely, in Proposition~\ref{pro:simply-connected fibers}).

 \strut

Let us sketch the proof of Theorem~\ref{th:main} and the plan of the article.  Let $X$ be as in the hypothesis of the statement,  having isolated singularity, and assume that the first integral $f$ of $X$ is such that its singular locus $\Sing(df)$ does not have one-dimensional real components.
Using a Brunella's result \cite{Bru} which  guarantees  that $X$ has a non-trivial orbit accumulating to the origin, we may assume, moreover, that the special fiber $Z=f^{-1}(f(0))$ of $f$ is not reduced to the single point $0\in\R^3$.
Under these assumptions, we prove, in Section~\ref{sec:fibers}, a technical result (Proposition~\ref{pro:simply-connected fibers}) which can be framed in the context of real versions of Milnor's Fibration Theorem \cite{Mil}.
Roughly speaking, it asserts that, in any sufficiently small neighborhood of the origin, $f$ has regular fibers with connected components which are simply connected and which accumulate to  a given  two-dimensional component of the special fiber $Z$.
Our proof of Proposition~\ref{pro:simply-connected fibers} requires some avatars of known results in the theory of reduction of singularities of analytic functions. We recall them in the form needed for our purposes.

In Section~\ref{sec:index}, we define, for any two-dimensional component $L$ of $Z$, the index $I_L(X)$ of the restriction $X|_L$, a generalization to singular surfaces of the usual notion of {\em Poincar\'{e} index} of a planar vector field  at a singular point. It is not really a new notion, it corresponds in one or another equivalent way to a particular case of  standard definitions of the index   of a vector field in a singular invariant variety (see \cite{Bra-Se-Su} for more information). Pushing the restricted vector field $X|_L$ to nearby fibers, using homotopic invariance of the index and the aforementioned result about simply connected fibers, we show that $I_L(X)$ is equal to zero for at least one component $L$.

In Section~\ref{section-proof}, we conclude the proof of  the first part of  Theorem~\ref{th:main} proving that, given a two-dimensional component $L$ of $Z$, either there exists a formal separatrix of $X$ inside $L$ or $I_L(X) \neq 0$ (Proposition~\ref{pro:index-separatrix} below). Incidentally, we use again the reduction of singularities as presented in Section~\ref{sec:fibers} for the proof of this result. As mentioned before, it  generalizes a known result of planar vector fields to the situation of vector fields in singular surfaces. It is related to Bendixson's formula for the computation of
the Poincar\'{e} index using hyperbolic and elliptic sectors of the vector field at the singularity.

Finally, in Section~\ref{sec:example}, we provide an explicit example of a vector field $X$ with isolated singularity at $0\in\R^3$ which has an analytic first integral but which does not have any real analytic separatrix.  The difficult part to check is that the formal real separatrix of such an example does really diverge. For that, we  use  the Martinet-Ramis moduli of planar holomorphic foliations of saddle-node type \cite{Mar-R, Ily} and the  computation of the tangent of the moduli map in Elizarov's work \cite{Eli}. We  thank Lo\"{\i}c Teyssier for his comments and decisive remarks concerning these arguments and techniques.

\section{About the fibers of a real analytic function} \label{sec:fibers}
 The main result in this section is  Proposition~\ref{pro:simply-connected fibers} below, a result on the geometry of the fibers
of a real analytic function in $\R^3$.
We provide a proof adapted to our situation which employs the reduction of singularities  of analytic functions. Some of the arguments are inspired on those of the paper \cite{Gra-S}
 and also on a part of Roche's work \cite{Roc} concerning {\em Real Clemens Structures}.

Our starting point is the following result (see Aroca-Hironaka-Vicente
\cite{Aro-H-V}, Hironaka \cite{Hir} or Bierstone-Milman
\cite{Bie-M1,Bie-M2}).
\begin{theorem}\label{th:reduction-singularities}
Let $f:(\R^n,0)\to(\R,0)$ be a non-zero real analytic function. There exists a neighborhood $U$ of $0\in\R^n$ and a sequence of
blow-ups with closed analytic non-singular centers
\begin{equation}\label{eq:blow-ups}
\pi:M_m\stackrel{\pi_m}{\longrightarrow}M_{m-1}\stackrel{\pi_{m-1}}{\longrightarrow}
\cdots  \stackrel{\pi_2}{\longrightarrow}M_1\stackrel{\pi_1}{\longrightarrow}U
\end{equation}
 such that
$f\circ\pi:  M_m \to \R$ is everywhere locally of monomial type, i.e. it
can be written locally as a monomial times a unit in analytic
coordinates. Moreover, if $Y_{j-1}$ is the center of $\pi_j$ for
$j=1,...,m$, and we define recursively the total divisor $E_j$ at
stage $j$ by $E_j=\pi_{j}^{-1}(E_{j-1} \cup Y_{j-1})$ with $E_0= \emptyset$, then
$Y_j$ has normal crossing with $E_j$ and it is contained in the
 singular  locus   $\Sing(df_j)$  of $f_{j}=f \circ \pi_{1}\circ\cdots\circ\pi_j$,  for $j\geq 0$,
where $f_0=f$.
\end{theorem}

In particular, if $Z=f^{-1}(0)$ and $\tilde{Z} =Z\setminus\Sing(df)$  is assumed to be non-empty (thus $\tilde{Z}$ is a smooth analytic hypersurface), then $\pi$ restricts to an analytic isomorphism from $\pi^{-1}(\tilde{Z})$ to $\tilde{Z}$.

For our purposes, we will use {\em real blow-ups} instead of the
usual (projective) blow-ups $\pi_j$ in Theorem~\ref{th:reduction-singularities}.
In order to define properly a real blow-up, we must consider the
category of real analytic manifolds with boundary and corners; i.e.
manifolds locally defined in coordinate charts $(x_1,...,x_n)$ as
quadrants $\{x_{i_1}\geq 0,x_{i_2}\geq 0,\ldots,x_{i_r}\geq 0\}$ and
so that the changes of coordinates are analytic isomorphisms preserving the quadrants. The point is that a real blow-up (also called a ``polar blow-up'') produces a boundary in the blown-up manifold, namely   the inverse image of the center by the blow-up  (called   \emph{exceptional divisor}), which corresponds to the set of   {\em half-directions} (instead of  directions)  in the normal bundle of  the center as a submanifold of the ambient space. Subsequent real blow-ups produce new boundaries which intersect old boundaries along  corners.

Let us  recall the main definitions  here   (see for instance the recent reference
\cite{Mar-R-S} for details). First, we define the real blow-up,  with closed non-singular center $Y$, on a real manifold without boundary $M$. Let $\pi:M_1\to M$ be the usual blow-up of $M$ with center $Y$ and let $\tau:M_1^+\to M_{1}$ be the orientable double covering of $M_1$. The composition $\pi\circ\tau$ is an analytic map which ramifies along the divisor  $E=\pi^{-1}(Y)$. Then the real blow-up of $M$ with center $Y$ is the restriction $\sigma:M'_1\to M$ of  $\pi\circ\tau$  to only one sheet, so that $M'_1$ is an analytic manifold with boundary $\partial M'_1=E$. Next, more generally, if $M$ is a real analytic manifold with boundary and corners and $Y\subset M$ is a non-singular analytic submanifold having normal crossings with $\partial M$, we may consider first
 $M$ immersed in a real analytic manifold $\wt{M}$ with no boundaries or corners of the same dimension (the immersion is locally uniquely determined up to analytic isomorphisms), so that $\partial M$ becomes a normal crossing divisor of $\wt{M}$ and such that $Y$ is sent into a non-singular submanifold $\wt{Y}\subset\wt{M}$ with normal crossings with $\partial M$ inside $\wt{M}$. The real blow-up
$\sigma:M'\to M$ with center $Y\subset M$ is the restriction of the
real blow-up $\wt{\sigma}:\wt{M}'\to\wt{M}$ with center $\wt{Y}$ to
$M'=\,\overline{\wt{\sigma}^{-1}(M\setminus Y)}$.

\strut

With this construction in mind, we    adapt Theorem~\ref{th:reduction-singularities} to obtain a version which uses   real blow-ups  and  which will be  more convenient  for us. Although we can consider general statements, we will concentrate  on  three-dimensional analytic functions with some extra condition concerning its singular locus.

 Fix a germ $f:(\R^3,0)\to(\R,0)$ of   analytic function.  Consider the prime decomposition $f=f_1^{n_1}f_2^{n_2}\cdots f_r^{n_r}$, where each $f_j$ is an irreducible germ of analytic function, and let $h=Red(f)=f_1f_2\cdots f_r$. Notice that $Z=f^{-1}(0)=h^{-1}(0)$. Assume the following property, that we   call {\em Reduced Isolated Singularity}:
 \begin{quote}
  {\bf  (RIS).-} The germ of analytic set $Z=f^{-1}(0)$ is not reduced to  $\{0\}$ and $\Sing(dh)\subset\{0\}$.
 \end{quote}
Note that the hypothesis (RIS) implies that, in some neighborhood of the origin, the set $Z\setminus\{0\}$ is a non-singular two-dimensional analytic submanifold and that the irreducible components $f_j^{-1}(0)$ of $Z$, as germs of   analytic sets, only intersect at $0$. (The converse of this result is not true: take $f=Red(f)=y^3-x^6$ for which the special fiber $Z=\{y-x^2=0\}$ is a non-singular surface at every point and the $z$-axis is contained in $\Sing(df)$.)
To be more precise, let
  $\eps>0$ be sufficiently small such that $f$ is defined and analytic in a neighborhood of the closed ball $V=\overline{B}(0,\eps)$, and such that $Z\cap V$ cuts transversally the boundary of $V$.  By
the Conic Structure
Theorem (see Milnor \cite{Mil} or vdDries \cite{vdD} for a more
general statement),  the set  $(Z\setminus\{0\})\cap V$ has finitely many connected components, denoted by $L_1, L_2,..., L_r$, where each $L_i$ is a non-singular analytic
surface immersed in $V$ whose closure in $V$ is homeomorphic to the
cone at $0$ over the link $C_i= \partial V \cap L_i$
(a curve homeomorphic to $\SSS^1$). The germs of the components $L_i$ at $0$ are well defined and do not depend on $\eps$. We will use the same notation $L_i$ for both the components of $(Z\setminus\{0\})\cap V$ (for any given sufficiently small $\eps$) and their germs. They will be called   {\em local components} of the special fiber $Z=f^{-1}(0)$.

\begin{proposition}
\label{cor:real-reduction}
Let $f:(\R^3,0)\to(\R,0)$ be a germ of  analytic function  that satisfies the hypothesis (RIS). Then, if $\eps>0$ is sufficiently small and $V=\overline{B}(0,\eps)$, there is a sequence of real blow-ups (independent of $\eps$)
\begin{equation}\label{eq:real-blow-ups}
\sigma:M'_m\stackrel{\sigma_m}{\longrightarrow}M'_{m-1}
\stackrel{\sigma_{m-1}}{\longrightarrow} \cdots \stackrel{\sigma_2}{\longrightarrow}
M'_1\stackrel{\sigma_1}{\longrightarrow}V,
\end{equation}
 such that the composition $f\circ\sigma$ is everywhere locally of monomial type and such that,
if $L$ is a local component of   $Z=f^{-1}(0)$, we have:
\begin{enumerate}[(i)]
  \item $\sigma^{-1}(L)$ is diffeomorphic to the half-open cylinder $[0,1)\times\mathbb{S}^1$, where the boundary $\{0\}\times\mathbb{S}^1$ corresponds to the link $C=L\cap \partial V$.
  \item The strict
  transform $L'=\overline{\sigma^{-1}(L)}$ of $\overline{L}=L\cup\{0\}$ is a
  real analytic submanifold of $M'_m$ with boundary and corners,
  homeomorphic to the closed cylinder $[0,1]\times\SSS^1$.
  \item  Denoting $\partial L'=C_\infty\cup\sigma^{-1}(C)$ the two connected components of the boundary of $L'$, we have that  $L'$
cuts transversally the total divisor $E'_m$
  along  $C_\infty$, which is a piecewise smooth
  analytic curve homeomorphic to $\SSS^1$.
  \item The strict transforms of two different local components do not intersect.
  \end{enumerate}
Moreover, $\sigma_1$ is the real blow-up with center $Y'_0 = \{0\}$ and, if
$Y'_{j-1}$ is the center of $\sigma_j$ for $j=2,...,m$, and we define
recursively the total divisor $E'_j$ at stage $j$ by
$E'_j=\sigma_{j}^{-1}(E'_{j-1} \cup  Y'_{j-1}  )$ with $E'_0= \emptyset $, then,  for any $j\geq 1$,
 $Y'_j\subset E'_j$ and $E'_j$ is homeomorphic to the
sphere $\SSS^2$.
\end{proposition}
\begin{proof}
Let $f=f_1^{n_1}f_2^{n_2}\cdots f_r^{n_r}$ be the prime decomposition of $f$ as a germ and put $h=Red(f)=f_1f_2\cdots f_r$. Let $\eps>0$ be sufficiently small such that $f$ is defined in a neighborhood of a closed ball $V=\overline{B}(0,\eps)$, and such that $Z$ cuts transversally the boundary of $V$.
Assume moreover that $V$ is contained in a
neighborhood where Theorem~\ref{th:reduction-singularities}
applies to $h$, so that we obtain a sequence of blow-ups $\pi$ as  in
(\ref{eq:blow-ups}) with centers $Y_0,Y_1,...,Y_{m-1}$, such that $ h \circ\pi$ is everywhere locally of monomial type.  Therefore,  the composition $f\circ\pi$ is also everywhere locally of monomial type.
Define the
sequence (\ref{eq:real-blow-ups}) recursively as follows: $\sigma_1:M'_1\to
V$ is the real blow-up of $V$ with center $Y'_0=Y_0$, $\sigma_2:M'_2\to
M'_1$ the real blow-up with center $Y'_1=Y_1\cap M'_1$, and so on.
 Since  $\Sing(dh)\subset\{0\}$, by the hypothesis (RIS),  and since the center
$Y_{j-1}$ of $\pi_j$ is contained in the singular locus of
$ h_{j-1}=h \circ \sigma_{1}\circ\cdots\circ\sigma_{j-1}$, we have that $Y'_0=\{0\}$ and that $Y'_j\subset E'_j$ for $j\geq 1$. We deduce then that
$E'_j\cong\SSS^2$ by recurrence on $j$, using the definition of real blow-up.

 Property \emph{(i)} is a consequence of the mentioned Conic Structure
Theorem, together with
 the fact that $\sigma:M'_{m} \setminus E'_{m} \to V \setminus \{0\}$ is
a diffeomorphism since each center $Y'_{j}$ is contained in $E'_{j}$  for $j\geq 0$.
To prove properties \emph{(ii)} and \emph{(iii)} we use the conclusion that
$\pi^{-1}(Z\cap V)$ is a normal crossing divisor,  so that
 $L'$ is contained in one of its components, a
non-singular analytic surface which cuts transversally the components of the total divisor $E'_m$. Finally, for property \emph{(iv)}, notice that if
 $L'_1,L'_2$ are the strict transforms of two different local components $L_1,L_2$ of $Z$ and
$L'_1\cap L'_2  \neq \emptyset$, then necessarily $L'_1\cap L'_2\not\subset E'_m$ (since $L'_1,L'_2$ and any component of $E'_m$ are components of the normal crossing divisor $\sigma^{-1}(Z\cap V)$). Hence $\sigma^{-1}(L_1)\cap\sigma^{-1}(L_2)\neq\emptyset$ and also $L_1\cap L_2\neq\emptyset$, which is impossible by the hypothesis (RIS).
\end{proof}
\begin{proposition}
\label{pro:simply-connected fibers}
Let $f:(\R^3,0)\to(\R,0)$ be a germ of   analytic function  satisfying the hypothesis (RIS).  Then there is a local component $L$ of the special  fiber $Z$ and a neighborhood base $\mathcal{B}$ of
$0\in\R^3$ such that each $U\in\mathcal{B}$  is compact and satisfies the  following property: there exists a family  $\{F^U_\lambda\}_{\lambda\in(0,\delta)}$, where $F^U_\lambda$ is a   connected component of a non-singular fiber of $f|_U$, such that $F^U_\lambda$ is homeomorphic to a closed disc and such that
$F^U_\lambda \xrightarrow{\lambda\rightarrow 0} (L\cup\{0\})\cap U$
in the Hausdorff topology.
\end{proposition}
\begin{proof}
We prove that any closed ball $V=\overline{B}(0,\eps)$, with $\eps>0$ sufficiently small for which Proposition~\ref{cor:real-reduction} holds, contains a neighborhood $U$ with the required properties of the statement.
We use  notation  of Proposition~\ref{cor:real-reduction}
so that, if $L_1,...,L_r$ are the local components of the singular fiber $Z=f^{-1}(0)$ and
$L'_j$ is the strict transform of $\overline{L_j}$,
then $L'_j$ is homeomorphic to the cylinder $[0,1]\times\SSS^1$ and  $L'_j\cap E'_m$ is a curve homeomorphic to $\SSS^1$.  Moreover, $L'_i\cap L'_j\cap E'_m=\emptyset$ if $i\neq j$. Let $j_0$ be such that  one of the connected components of $E'_m\setminus L'_{j_0}\cap E'_m$, say $D$, contains no curve $L'_j\cap E'_m$ for $j\neq j_0$. Then
$\Omega=L'_{j_0}\cup D$ is homeomorphic to a closed disc.

\medskip


\begin{center}
\begin{overpic}[scale=.7]
{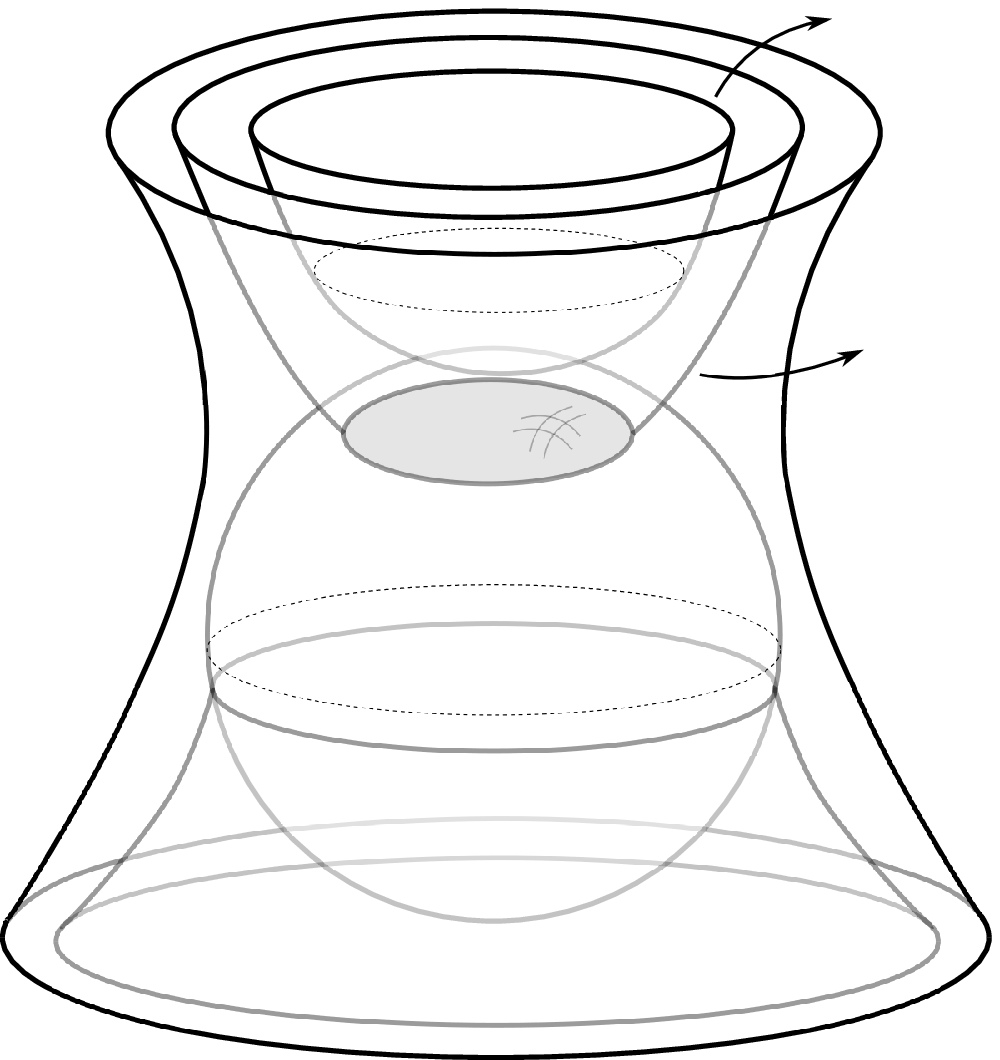}
\put (83,124) { $D$}
\put (173,141) { $L_{j_{0}}'$}
\put (170,208) {$\Delta_{\delta}$}
\end{overpic}
\end{center}

\medskip

Let us prove the statement for the component $L=L_{j_0}$.
Consider
$\wt{f}=f\circ\sigma:M'_m\to\R$, whose singular fiber is given by
$$
\wt{Z}=\wt{f}^{-1}(0)=\sigma^{-1}(Z\cap V)=L'_1\cup L'_2\cup \cdots \cup L'_r\cup E'_m,
$$
and thus $\Omega\subset\wt{Z}$. By construction, there is a unique connected component of $M'_m\setminus\wt{Z}=\sigma^{-1}(V\setminus f^{-1}(0))$,  denoted by $K$, whose  topological frontier  in $M'_m$ is exactly $\Omega$. We assume, without lost of generality, that $\wt{f}$ is positive on $K$.  Denote also by
$\dot{M}'_m=\sigma^{-1}(V\setminus\partial V)$ (a manifold with boundary where $\partial\dot{M}'_m=E'_m$).

 Let $g$ be an analytic
riemannian metric on $M'_m$ (the existence of such a metric is
guaranteed by Grauert's Analytic Immersion Theorem \cite{Gra}) and let $\xi=-\nabla_g(\wt{f}^2)$ be
the gradient vector field of $\wt{f}^2$ with respect to $g$. The
square and the sign ``$-$'' are taken  in order  to guarantee both that $\wt{f}$
decreases along any trajectory of $\xi$ and that  $\wt{Z}$ is
exactly the singular locus of $\xi$.  By a \L ojasiewicz's result
(see \cite{Loj}),  there exists an open neighborhood $H$ of $\wt{Z}\cap \dot{M}'_m$ in $\dot{M}'_m$ such that for any $p\in H$, the  integral curve $\g_p$ of $\xi$ with $\g_p(0)=p$ is defined on $[0,\infty)$ and the limit
$$
R_\xi(p)=\lim_{t\to\infty}\g_p(t)
$$
exists and belongs to $\wt{Z}\cap\dot{M}'_m$. Moreover, the map $R_\xi:H\to\wt{Z}\cap\dot{M}'_m$ is a continuous retraction.

  In our particular case where $\wt{f}$ is locally of monomial type, one can show, moreover, the following:

  \begin{claim}
  For any $q\in\Omega\cap  \dot{M}'_m$, there exists a neighborhood $B_q$ of $q$ and a unique orbit of $\xi$ in $B_q\cap K$ that accumulates to $q$.
  \end{claim}

Assume that the Claim is true. Put $\Omega'=\Omega\cap\sigma^{-1}(\overline{B}(0,\eps/2))$.
Using the existential part of the Claim, the fact that $\wt{f}$ decreases along integral curves
of $\xi$ and the compactness  of $\Omega'$, there exists a  fiber $\Delta_\delta=\wt{f}^{-1}(\delta) \cap K$
of $\wt{f}|_K$, for some $\delta>0$, such that  $\Omega' \subset R_\xi(\Delta_\delta)$.
(Notice that fibers of $\wt{f}|_K$ and of $\wt{f}^2|_K$ coincide since we have assumed that
$\wt{f}$ is positive on $K$).
On the other hand,   by  the uniqueness property stated in the Claim
 and  since  an orbit  of $\xi$ can intersect at most once any fiber of $\wt{f}$, if $\wt{F}_{\delta}=R_\xi^{-1}(\Omega')\cap \Delta_\delta$  then $$R_\xi|_{\wt{F}_{\delta}}:\wt{F}_{\delta}\to\Omega'$$
 is bijective, and hence a homeomorphism.  Observe that all the conclusions above also hold for any $\lambda$ with $\lambda \in (0,\delta]$, using the flow of $\xi$, which provides, by restriction, a diffeomorphism from $\wt{F}_\delta$ to $\wt{F}_{\lambda}$   for  every such $\lambda$. In particular, $\wt{F}_{\lambda}=R_\xi^{-1}(\Omega')\cap \Delta_{\lambda}$   is homeomorphic to a closed disc for  any $\lambda \in (0,\delta]$.

 We finally consider   the set
$$
\wt{U}=( M'_m \setminus K)\cup(R_\xi^{-1}(\Omega')\cap
 \wt{f}^{-1}([0,\delta])\cap\overline{K}
),
$$
which is a  compact  neighborhood of the total divisor $E'_m$ in $M'_m$, and we put $U=\sigma(\wt{U})$.  Then $U$ is a neighborhood of the origin contained in $V$ with the required properties  for the family $\{F^U_\lambda=\sigma(\wt{F}_\lambda)\}_{\lambda\in(0,\delta]}$.

\medskip

\noindent {\em Proof of the Claim.}
Fix $q\in\Omega\cap \dot{M}'_m$ and denote by $e=e(q)$ the number of components of $\wt{Z}$, considered  as a normal crossing divisor, which meet at $q$. We analyze separately the three possible values of $e \in \{1,2,3\}$. First, it is worth   recalling  that the expression of the vector field $\xi=-\nabla_g\wt{f}^2=A\partial/\partial x+B\partial/\partial y+C\partial/\partial z$ in analytic coordinates $w=(x,y,z)$ at $q$ is computed by the formula
\begin{equation}\label{eq:expression-xi}
\left(A\; B\; C\right)=
-\frac{\partial\wt{f}^2}{\partial w}
\left( h_{ij}
\right),
\end{equation}
where $(h_{ij})$ is the inverse of the matrix of the metric $g$ in the coordinates $w$ and $ \partial\wt{f}^2 / \partial w $ is the row vector of partial derivatives of $\wt{f}^2$.

\strut

{\em Case $e=1$.} We choose an analytic chart $(B_q,(x,y,z))$ centered at $q$ so that $\wt{f}=x^m$, with $m>0$, and $B_q\cap K=\{x>0\}$. Inside the domain $B_q$ of the chart, using (\ref{eq:expression-xi}), we may write $\xi=x^{2m-1}\bar{\xi}$, where $\bar{\xi}$ is a vector field, which is  non-singular at $q$. Moreover, $\bar{\xi}$ is transversal to $\wt{Z}=\{x=0\}$ in a neighborhood of $q$. Thus, the orbit of $\bar{\xi}$ through $q$ is the unique orbit that may accumulate to $q$ and   cuts $\{x>0\}$. Since orbits of $\xi$ in $\{x>0\}$ are contained in  orbits of $\bar{\xi}$, we conclude the claim.

\strut

{\em Case $e=2$.} In this case, we choose an analytic chart $(B_q,(x,y,z))$ such that $\wt{f}=x^my^n$, with $m,n>0$, and $B_q\cap  K=\{x>0,\,y>0\}$. Then, using (\ref{eq:expression-xi}), we may write $\xi=2x^{2m-1}y^{2n-1}\bar{\xi}$, where
\begin{equation}
\label{grad-field}
\bar{\xi}=-(myh_{11}+nxh_{21})\frac{\partial}{\partial x}-(myh_{12}+nxh_{22})\frac{\partial}{\partial y}.
\end{equation}
Since   the  orbits of $\xi$ and $\bar{\xi}$ coincide in $B_q\cap K$, it suffices to prove the Claim for $\bar{\xi}$. Using the fact that $(h_{ij})$ is positive definite, we have that $\Sing(\bar{\xi})=\{x=y=0\}$ and that the linear part $D\bar{\xi}(q)$ of $\bar{\xi}$ at $q$ has (real) eigenvalues $\{0,\lambda,\mu\}$, where $\lambda<0<\mu$.
Let $W^s,W^u$ be the stable and unstable manifolds of $\bar{\xi}$ at $q$. They are invariant smooth curves (in fact real analytic separatrices of $\bar{\xi}$, see \cite{Car-S}) tangent to the eigendirections $E_\lambda,E_\mu$ corresponding to $\lambda$ and $\mu$, respectively. Also, $W^c=\Sing(\bar{\xi})$ is a center manifold  of $\bar{\xi}$ at $q$. The {\em Theorem of  Reduction to the Center Manifold} (see \cite{Hir-P-S,Carr}) implies that $\bar{\xi}$ is topologically equivalent, in a neighborhood of $q$, to the linear vector field $D\bar{\xi}(q)=\lambda u\partial/\partial u+\mu v\partial/\partial v$ in $\R^3$, where $u$ and $v$ are linear  coordinates on $E_\lambda$ and $E_\mu$, respectively.
As a consequence, the four connected components of $(W^s\cup W^u)\setminus\{q\}$ are the unique non-trivial orbits of $\bar{\xi}$ which accumulate to $q$.
It suffices to show that $(W^u\cup W^s)\cap K =W^s\cap K\neq\emptyset$ (notice that in that case only one of the components of $W^s\setminus\{q\}$ may be contained in  $K$, since $W^s,W^u$ are transversal  at $q$ to the  components $\{x=0\}$ and $\{y=0\}$ of $\wt{Z}$,  both contained in $Fr(K)$).

 Let us  show  that $W^u\cap  K=\emptyset$. Notice that the integral curve of $\bar{\xi}$ at any point of $W^u\setminus\{q\}$ is defined in an interval of the form $(-\infty,a)$ and converges to $q$ for $t\to - \infty$. This would also be the case for an integral curve of $\xi$ if $W^u\cap K \neq\emptyset$, since the sense of parametrization of integral curves of $\xi$ and $\bar{\xi}$ coincide in $B_q\cap K$. However, this is impossible because $-\wt{f}^2$ grows along integral curves of $\xi$ in  $K$ and $\wt{f}(q)=0$.

 Let us show  that $W^s\cap H\neq\emptyset$. Denote by $\Delta=E_\lambda\oplus E_\mu$, a linear plane invariant  by  the linear vector field $D\bar{\xi}(q)$. Let $Q$ be the cone inside $\Delta$ bounded by the half-lines $\ell_x,\ell_y$  of $\Delta$ which correspond to the tangent directions of  $\{y=0,x\geq 0\}\cap\Delta$ and $\{x=0,y\geq 0\}\cap\Delta$, respectively. If $W^s\cap  K =\emptyset$ then we would have $Q\cap(E_\lambda\cup E_\mu)=\{0\}$. In this case, we could see that the vector field $D\bar{\xi}(q)$,  that is everywhere transversal to  the boundary of $Q$, {\em enters}   $Q$ through one of  the half-lines $\ell_x,\ell_y$  while it {\em escapes} from $Q$ through the other one. This is impossible by comparing  $D\bar{\xi}(q)$ with $\bar{\xi}$, since this last vector field  escapes from $K$ through any point of $Fr(K)\setminus\{x=y=0\}=\{y=0,x>0\}\cup\{x=0,y>0\}$.

\strut

{\em Case $e=3$.} We use the result  by Kurdyka et al. \cite{Kur-M-P}  that solves Thom's Con\-jec\-ture  : Let $h:(\R^n,0)\to\R$ be an analytic function and let $g$ be a real analytic riemannian metric at $0$. Then any non-trivial orbit $\G$ of  the  analytic gradient vector field $\nabla_gh$ that converges to  $0\in\R^n$ has a well defined  limit  tangent
$$
\nu_\G=\lim_{x\in\G,x \to 0}\frac{x}{\|x\|}\in\mathbb{S}^{n-1}.
$$
Also, we use the following results from Moussu's paper \cite[Theorems 1 and 3]{Mou}: 
if $g(0)$ is the Euclidean metric in $\R^n=T_0\R^n$ and $H_k=h_k|_{\mathbb{S}^{n-1}}$, where $h_k$ is the first non-zero homogeneous polynomial (of degree $k$) in the Taylor expansion of $h$ at $0\in\R^n$, then we have:
\begin{enumerate}[(a)]
  \item If $\nu_\G\in\SSS^{n-1}$ is the  limit  tangent of an orbit $\G$ of $\nabla_gh$ that converges to the origin, then $\nu_\G\in\Sing(d H_k)$.\footnote{This result, essentially, was already established by Martinet and Thom.}
  \item Assume that $h\geq 0$ in a neighborhood of the origin and denote by $S_0\subset\SSS^{n-1}$ the set of points $\nu\in\SSS^{n-1}$ that satisfy
      $$
      \nu\in\Sing(dH_k),\;\;H_k(\nu)<0,\;\;kH_k(\nu)<\inf\{\lambda_1(\nu),...,\lambda_{n-1}(\nu)\},
      $$
      where the $\lambda_j(\nu)$ are the eigenvalues of the hessian matrix of $H_k$ at $\nu$ (with respect to the standard metric in $\SSS^{n-1}$).  Then, for  any $\nu\in S_0$ there exists a unique orbit $\G$ of $\nabla_gh$ converging to the origin such that $\nu=\nu_\G$ (in fact, $\G$ is an analytic separatrix of $\nabla_gh$).
\end{enumerate}
In order to apply these results to our  gradient vector field  $\xi=\nabla_g(-\wt{f}^2)$, we consider an analytic chart $(B_q,w=(x,y,z))$, centered at $q$, such that the matrix of the metric $g$ in the coordinates $w$ at $w=0$ is equal to the identity and, moreover, $\wt{f}$ is written in the form
$
\wt{f}=\ell_1^{m_1}\ell_2^{m_2}\ell_2^{m_3},
$
where $\ell_1,\ell_2,\ell_3$ are  linearly independent homogeneous polynomials of degree one  in the variables $(x,y,z)$. To show  that we can choose such an analytic chart,    first take coordinates $\bar{w}=(\bar{x},\bar{y},\bar{z})$ so that $\wt{f}=\bar{x}^{m_1}\bar{y}^{m_2}\bar{z}^{m_3}$ and then take a linear change of variables $w=P\bar{w}$ so that $g(0)$ has the identity matrix in the coordinates $w$.
We may assume also that   $K\cap B_q$    is described by the set  $\{\ell_1,\ell_2,\ell_3>0\}$  and that the range $w(B_q)$ of the chart contains a neighborhood of $[-1,1]^3$, so that the sphere $\SSS^2=\{x^2+y^2+z^2=1\}$ in the coordinates $w$ is well defined as a ``sphere'' inside $B_q$.  Consider
$$
F=(-\ell_1^{2m_1}\ell_2^{2m_2}\ell_2^{2m_3})|_{\SSS^2},
$$
 an  analytic non-constant function on $\SSS^2$. According to Moussu's results (a) and (b) above, it suffices to prove that:
\begin{enumerate}[(i)]
  \item $F$ has a unique singular point  $\nu_0$  in $\SSS^2\cap K$, which is a local minimum for $F$ (thus the hessian of $F$ at  $\nu_0$  is  positive semidefinite  and hence  $\nu_0$ belongs to the set $S_0$  defined in (b) above).
  \item No point of the frontier of $\SSS^2\cap K$ in $\SSS^2$ can be the limit  tangent of an orbit of $\xi$ contained in $K$.
\end{enumerate}
 Property  (i)  is an exercise in convex geometry:  For any $c>0$, the function $\tilde{f}^{2}-c=\ell_1^{2m_1}\ell_2^{2m_2}\ell_3^{2m_3}-c$ restricted to $K=\{\ell_1,\ell_2,\ell_3>0\}$ is such that its epigraph $\{\tilde{f}^{2}\geq c\}\cap K$ is strictly convex. Thus, if $\nu_0\in\mathbb{S}^2$ is a singular point of $\tilde{f}^{2}|_{\mathbb{S}^2\cap K}$ then the tangent plane of $\mathbb{S}^2$ at $\nu_0$ equals the tangent plane of the fiber $(\tilde{f}^{2})^{-1}(\tilde{f}(\nu_0)^{2})$ at $\nu_0$ and separates $\mathbb{S}^2$ from the epigraph $\{\tilde{f}^{2}\geq\tilde{f}(\nu_0)^{2}\}$. Thus $\nu_0$ is a global maximum of $\tilde{f}^{2}$ in restriction to $\mathbb{S}^2\cap K$, which shows (i).


 Let us show (ii). The set $T=\SSS^2\cap\overline{K}$ is a  spherical triangle $T$ determined by the lines $\ell_j\cap\SSS^2$, $j=1,2,3$. We consider the real blow-up $\sigma_q:\wt{M}\to M'_m$ at $q$ so that the divisor $E=\sigma_q^{-1}(q)$ is identified with the sphere $\SSS^2$. The transformed vector field $\sigma_q^*\xi$ is singular along $E$ but it can be divided by an equation of $E$ so that we obtain a new vector field $\wt{\xi}'$ on $\wt{M}$, which leaves invariant the divisor $E$ and so that $E\not\subset\Sing(\wt{\xi}')$. A calculation (which, this time, is easier assuming that the coordinates are chosen so that $\wt{f}=x^{m_1}y^{m_2}z^{m_3}$) shows that  $\Sing(\wt{\xi}')\cap T$ is the set of vertices of $T$.  This proves that if an orbit of $\wt{\xi}'$, not contained in the divisor $E$, accumulates to a single point  of $T$, then this point must be a vertex. On the other hand,  if $v$  is a vertex of $T$, one can see that $\rm{Sing}(\wt{\xi}')$ is a non-singular curve at $v$ transversal to the divisor $E$ and that the restriction of $\wt{\xi}'$ to $E$ is a linear  vector field (in standard charts for the blow-up) with real eigenvalues of different sign. Thus the stable and unstable manifolds of $\wt{\xi}'$ at $v$ are contained in $E$, whereas  $\rm{Sing}(\wt{\xi}')$ is a center manifold. Using the Theorem of Reduction to the Center Manifold in a way analogous to the case $e=2$, we conclude  that no orbit of $\wt{\xi}'$ outside $E$ can accumulate to $v$. This proves  (ii), as wanted.
\end{proof}

\section{The Poincar\'e-Hopf  index}\label{sec:index}

Let $f:(\R^3,0)\to\R$ be a germ of   analytic function which satisfies the hypothesis (RIS) and let $\varepsilon>0$ be sufficiently small so that Proposition~\ref{cor:real-reduction} holds for $f$. Put $Z=f^{-1}(0)\cap\overline{B}(0,\varepsilon)$ and let $L$ be one of the  local components  of $Z$. Consider in $L$ the orientation induced by the normal vector field $\nabla f|_{L}$. Let $C = L \cap \partial \overline{B}(0,\eps) \cong \SSS^{1}$ be the corresponding link with  the usual orientation as a boundary of $L$.
By the conic structure of $L$, there exists a homeomorphism $\Psi: \overline{L} \to \D$, where $\D$ is the unit closed disc in $\R^2$ centered at the origin, such that $\Psi(0) =0$ and which restricts to a diffeomorphism  from $L$ into $\D\setminus\{0\}$. We can suppose that $\Psi$ is
orientation preserving. Moreover, we can suppose that the tangent map $T\Psi:TL\subset T\R^3\to T\R^2$ of $\Psi$ over $L$ is uniformly bounded  for the usual norm of tangent vectors of $\R^n$  (for instance, change $\Psi$ by $g(\|\Psi\|)\Psi$ where $g:[0,1]\to[0,1]$ is a convenable monotonic $C^1$-function). Thus
$\Psi_{*}(X|_{L})$ extends to a continuous vector field
 $\wt{X}_L$  on the disc $\D$ with isolated singularity at $0$.
We define the {\em index of $X$  along  $L$}, denoted by  $I_{L}(X)$, to be the Poincar\'e-Hopf index of  $\wt{X}_L$  at the origin  of $\R^2$.
It can be computed as the degree of the map
$  \wt{X}_L /||  \wt{X}_L ||:\SSS^{1} \to \SSS^{1}$.
It is well known that $I_{L}(X)$ does not depend on $\eps$ or on the homeomorphism $\Psi$ (as long as it satisfies the mentioned properties).

\begin{proposition}
\label{prop-index}
 Let $X$ be a real analytic vector field at $(\R^{3},0)$
having a real-analytic first integral $f$. Assume that $f$ satisfies the hypothesis (RIS) and that $X$ has an isolated singularity at $0\in\R^3$. Then there exists
a local component $L$ of   $Z=f^{-1}(0)$
such that $I_{L}(X) = 0$.
\end{proposition}
\begin{proof}
Take  a  local component $L$ of $Z$ satisfying the properties stated  in Proposition~\ref{pro:simply-connected fibers}, i.e., in every neighborhood of $L$ there are fibers of $f$ that have connected components which are simply connected. Assume,
 without loss of generality, that there are such fibers with positive values of $f$.  Consider a diffeomorphism  $\Psi:L\to\D\setminus\{0\}$   and the vector field $\wt{X}_L=\Psi_\ast(X|_L)$  as in the paragraph above, so that $I_L(X)$ is the Poincar\'{e}-Hopf index of $\wt{X}_L$ at $0\in\R^2$.

We shift the link $C= L\cap \partial\overline{B}(0,\eps)$ of $ L $ to  nearby fibers of $f$  in the following way. Notice first that, by the condition (RIS),  if $f=f_1^{n_1}\cdots f_r^{n_r}$ is the decomposition of $f$ in irreducible factors, there is a unique $j$ so that $f_j$ vanishes along $C$ and, if $k\neq j$, then $f_k(x)\neq 0$ for any $x\in C$. Let $m=n_j$. In a sufficiently small neighborhood $W$ of $C$ in $\R^3$, the function $\beta:W\to\R$ defined by
$$
\beta=  f_j\prod_{k\neq j}|f_k|^{\frac{n_k}{m}}
$$
satisfies ${\rm Sing}(d\beta)=\emptyset$, $W\cap L=\{\beta=0\}$ and $f|_W= \epsilon \beta^m$, where $\epsilon=\pm 1$. Up to changing the sign of $\beta$ we can assume that $\epsilon=+1$ (notice that $\epsilon=-1$ and $m$ even cannot occur since we suppose that $f$ takes positive values near $L$). Notice that  the fibers of $\beta$ are contained in the fibers of $f|_W$. Put $Y=\nabla \beta / \| \nabla  \beta\|^{2}$, a vector field which is
 transversal to the fibers of $f$  in $W$, in particular to  $L\cap W$. Moreover, if $\phi_{t}(x)$
is the flow of $Y$, we have   $\beta(\phi_{t}(x)) = t$ for every $x \in  L\cap W$ and every $t \in \mathbb{R}$ sufficiently
small. Let us denote by $L_{t}$ the fiber $f=  t^m$  in $\overline{B}(0,\varepsilon)$  and by $C_{t}$ the curve $\phi_{t}(C)$.
 There exists a small $\rho >0$ such that,
 for each fixed   $|t| < \rho$, the flow $\phi_{t}(x)$ defines a diffeomorphism $\Phi_{t}$ between
an open neighborhood $A$ of $C$ in $L$ and an open neighborhood $A_{t}$ of $C_t$ in $L_{t}$, which restricts to a diffeomorphism from $C$ to $C_t$. Moreover, if $t>0$,  $\Phi_t$ preserves
 the orientation  induced by the gradient $\nabla f$ on the fibers of $f$.

 If $|t| < \rho$, the map $\Psi_t = \Psi \circ \Phi_{t}^{-1}$ takes $A_{t}$ diffeomorphically
  into   a neighborhood of $\SSS^{1}$ in $\R^2$, sending $C_t$ to $\SSS^1$.
We define $ \wt{X}_{t}  = \Psi_{t\,\ast}(X|_{A_{t}})$.
 The map $s\mapsto  \wt{X}_{st}$, for $s\in[0,1]$, defines a homotopy between $\wt{X}_L=\wt{X}_0$ and $\wt{X}_{t}$. Moreover, if $t$ is sufficiently small, we may assume that $\wt{X}_{st}$ never vanishes over $\SSS^{1}$.
Thus we have
$$
I_L(X)=\mbox{degree }(\wt{X}_L/||\wt{X}_L|| :\SSS^{1} \to \SSS^{1})=\mbox{degree }(\wt{X}_t/||\wt{X}_t|| :\SSS^{1} \to \SSS^{1}).
$$
 By our choice of the local component $L$ and Proposition~\ref{pro:simply-connected fibers}, if $t>0$ is sufficiently small, $C_t$ is contained in a connected component of $L_t$ that is simply connected. Hence, the curve $C_t$ is the boundary of a submanifold $D_t$ in $L_t$ diffeomorphic to the unit disc $\D$ via a diffeomorphism $h:D_t\to\D$, which can be extended to a neighborhood of $C_t$ in $L_t$ and  satisfies $h(C_t)=\SSS^1$. On the one hand, the vector field $\xi=h_\ast(X|_{D_t})$ in $\D$ has Poincar\'{e} index equal to $0$, since it never vanishes.
On the other hand, such an index can be calculated as the degree of the map $\xi/||\xi||:\SSS^1\to\SSS^1$ and this is equal to the degree of $\wt{X}_t/||\wt{X}_t|| :\SSS^{1} \to \SSS^{1}$, since $\xi$ and $\wt{X}_t$ are related by the diffeomorphism $h\circ\Psi_t^{-1}$ defined in a neighborhood of $\SSS^1$ in $\R^2$. This concludes the proof of the Proposition.
\end{proof}

\section{Proof the main theorem}

\label{section-proof}

In this section we complete  the proof of  the first part of
Theorem~\ref{th:main}.

Let $X$ be a germ of   analytic vector field with an isolated singularity at $0\in\R^3$ having a non-constant analytic first integral $f:\R^3\to\R$ with $f(0)=0$. As mentioned in the introduction, we may assume that the singular locus $\Sing(df)=\{p\in\R^3\,:\,df(p)=0\}$ of $f$ has no components of real dimension equal to one at $0$. This is a consequence of the following result
(whose proof given in \cite{Mol} for the complex case generalizes, without changes, to the real case)
 and the fact that $X$ is tangent to the foliation given by $df=0$.

\begin{proposition}\label{pro:singular-set-invariant}
Let $Y$ be a germ of   real analytic vector field having an isolated singularity at $0\in\R^3$. Let $\omega$ be a germ of   real analytic integrable $1$-form at $0\in\R^3$ such that $\omega(Y)=0$. Then the one-dimensional components of $\Sing(\omega)=\{p\,:\,\omega(p)=0\}$ are invariant by $Y$.
\end{proposition}
 Under the hypothesis that $f$ is a first integral of $X$, we obtain  that the special fiber $Z=f^{-1}(0)$ of $f$ is not reduced to a single point, i.e., that
$
Z\setminus\{0\}\neq\emptyset.
$
This is a consequence of Brunella's result \cite{Bru} which asserts that $X$ has a non trivial orbit $\tau$ accumulating to the origin. In our case, we have necessarily that $\tau\subset Z$  and hence $Z\neq\emptyset$, since the orbits of $X$ are contained in the fibers of $f$.

Moreover, the following lemma implies that we may assume that the function $f$ satisfies the (RIS) hypothesis.

\begin{lemma}
If $f$ does not satisfy the (RIS) hypothesis then $X$ has a real analytic separatrix.
\end{lemma}
\begin{proof}
Consider the prime decomposition $f=f_1^{n_1}f_2^{n_2}\cdots f_r^{n_r}$ where the $f_j$ are two by two different irreducible germs of real analytic functions and let $h=Red(f)=f_1 f_2\cdots f_r$. Notice that $Z=f^{-1}(0)=h^{-1}(0)$. We have already shown that $Z\not\subset\{0\}$, so if $f$ does not satisfy the (RIS) hypothesis then $\Sing(dh)\not\subset\{0\}$. In this case, there exists a component $H$ of the analytic set $\Sing(dh)$ of positive (real) dimension accumulating to the origin. Necessarily $H\subset\Sing(dh)\cap Z$ and hence $H$ is one-dimensional, since $\Sing(dh)\cap Z$ has no component of codimension one. Indeed, in a neighborhood of the origin, a point $p$ belongs to $\Sing(dh)\cap Z$ if and only it satisfies at least one of the following conditions:
\begin{enumerate}[(i)]
\item
there is a pair of indices $i,j$, with $i \neq j$, such that $p \in \{f_{i} = f_{j} = 0\}$;
\item  there
 is an index $j$ such that $p \in \{f_{j} = 0\} \cap \Sing(d f_{j})$.
 \end{enumerate}
 It suffices to prove that $H$ is invariant by $X$. For this, we consider the real analytic $1$-form obtained by canceling the poles of the  logarithmic derivative of $f$:
$$ \omega_{f} = Red(f) \frac{d f}{f} = \sum_{j=1}^{r} n_{j} f_{1} \cdots \widehat{f_{j}} \cdots f_{r} df_{j}.
$$
We have that $\Sing(\omega_f)\cap Z=\Sing(dh)\cap Z$ since both sets are described by the same properties (i) and (ii) above, and thus $H$ is a one-dimensional component of $\Sing(\omega_f)\cap Z$. Finally, since $\omega_f(X)=0$, the set $\Sing(\omega_f)\cap Z$ (and hence $H$) is invariant by $X$ by Proposition~\ref{pro:singular-set-invariant}, as wanted.
\end{proof}

Assume now that the first integral $f$ satisfies the hypothesis (RIS), so that we can apply to $f$ the constructions and the results described in the preceding sections. In particular, let $L_1,...,L_r\subset \overline{B}(0,\eps)$ be the  local  components of
$Z=f^{-1}(0)$, where $\eps$ is sufficiently small, as defined in Section~\ref{sec:fibers} and let $I_{L_j}(X)$ be the index of $X$ along $L_j$ as defined in Section~\ref{sec:index}.

Theorem 1 is a consequence of Proposition~\ref{prop-index} and of  the following result:

\begin{proposition}
\label{pro:index-separatrix}
For any $j\in\{1,2,...,r\}$, either there is a formal real
separatrix of $X$ contained in $L_j$ or $I_{L_j}(X)\neq 0$.
\end{proposition}
\begin{proof}
Fix $j\in\{1,2,...,r\}$ and put for simplicity $L=L_j$, $C=C_j$ etc. Assume that $\eps$ is sufficiently small so that Proposition~\ref{cor:real-reduction} holds for $V=\overline{B}(0,\eps)$. That is, there exists a sequence of real blow-ups $\sigma:M'\to V$ such that $L'=\overline{\sigma^{-1}(L)}$ is a real analytic surface with boundary and corners, homeomorphic to a closed cylinder $[0,1]\times\SSS^1$, such that $\sigma$ induces a diffeomorphism between $\sigma^{-1}(L)$ and $L$. The boundary of $L$ consists of the two components $C'=\sigma^{-1}(C)$ (the transform of the link of $L$ by $\sigma$) and $D'=L'\cap E'$, where $E'=\sigma^{-1}(0)$ is the exceptional divisor of $\sigma$. While $C'$ is a smooth analytic curve, $D'$ is only piecewise smooth analytic. Denote by $J\subset D'$ the set of  corners of $D'$, i.e., the set of  points where $D'$ is not smooth. Consider in $L'$ the orientation induced from that of $L$ by $\sigma$.
Up to considering  another surface diffeomorphic to $L'$, we may assume that $L'$ is a submanifold with boundary and corners inside the euclidean plane  $\R^2$, with the standard orientation.

The transformed vector field $X'=\sigma^\ast (X|_L)$ in $L'\setminus D'$ defines a one-dimensional singular analytic foliation $\mathcal{F}'$ which can be extended analytically to $D'$ as an oriented foliation (i.e., at any point $p\in D'$, there is an analytic vector field  $X'_p$  in a neighborhood $V_p$ of $p$ in $L'$, with isolated singularities, generating $\mathcal{F}'$ and such that  $X'_p$  and $X'$ are equally oriented in $V_p\setminus D'$) whose set of singular points $\Sing(\mathcal{F}')$ is finite and contained in $D'$. Moreover, using Seidenberg's Theorem on reduction of singularities \cite{Sei}, and up to considering new blow-ups on $L'$ at points of $D'$, we can assume that any point of $\Sing(\mathcal{F}')$ is a simple singularity (that is, the eigenvalues $\lambda,\mu$ of the  corresponding  linear part are real and satisfy $\mu\neq 0$ and $\lambda/\mu\not\in\Q_{>0}$) and that any connected component of $D'\setminus J$ is either invariant for or everywhere transversal to $\mathcal{F}'$.

Suppose that there is no formal real separatrix of $X$ inside $L$. Then, at any point $p\in D'$, the formal separatrices of $\mathcal{F}'$ at $p$ (of a generator $X'_p$ of $\mathcal{F}'$) are contained in $D'$. In particular, any connected component of $D'\setminus J$ is invariant for $\mathcal{F}'$. Also taking into account that a simple singularity of a two-dimensional real vector field, with real eigenvalues, has exactly two transversal formal separatrices (both real, non-singular and  tangent to the corresponding eigendirections), we have necessarily that  $\Sing(\mathcal{F}')=J$ and that the only  formal  separatrices of $\mathcal{F}'$ at any $p\in J$ are the two components of $D'$ through the point  $p$  (thus, they are analytic separatrices). Notice that, since $D'$ is contained in the boundary of $L'$, there are exactly two connected components of $D'\setminus J$ locally at $p\in J$, each of them is part of one of the separatrices of  $\mathcal{F}'$  at $p$.
Each connected component $\ell$ of $D'\setminus J$ is a non-singular oriented leaf of $\mathcal{F}'$ , going from $\alpha(\ell)$ to $\omega(\ell)$, both points in $J$.  A singular point $p\in J$ is either a {\em sink}, a {\em source} or a {\em saddle}, depending if $\omega(\ell)=\omega(\ell')=p$, $\alpha(\ell)=\alpha(\ell')=p$ or the remaining cases, respectively, where $\ell,\ell'$ are the two components of $D'\setminus J$ which accumulate to $p$. Sinks an sources are jointly called {\em nodes}. A {\em node connection} is a union $$\tau=\overline{\ell_1}\cup\cdots\cup\overline{\ell_r}$$ where the $\ell_j$ are connected components of $D'\setminus J$ satisfying $\alpha(\tau):=\alpha(\ell_1)$ is a source, $\omega(\tau):=\omega(\ell_r)$ is a sink and $\omega(\ell_j)=\alpha(\ell_{j+1})$  for $j=1,...,r-1$ (which are saddle points). By construction, there is a continuum of trajectories of $X'$ in $L'\setminus D'$ accumulating to $\tau$ and having $\alpha(\tau)$, $\omega(\tau)$ as the $\alpha$ and $\omega$-limit set, respectively. See the figure below:


\medskip
\begin{center}
\begin{overpic}[scale=1]
{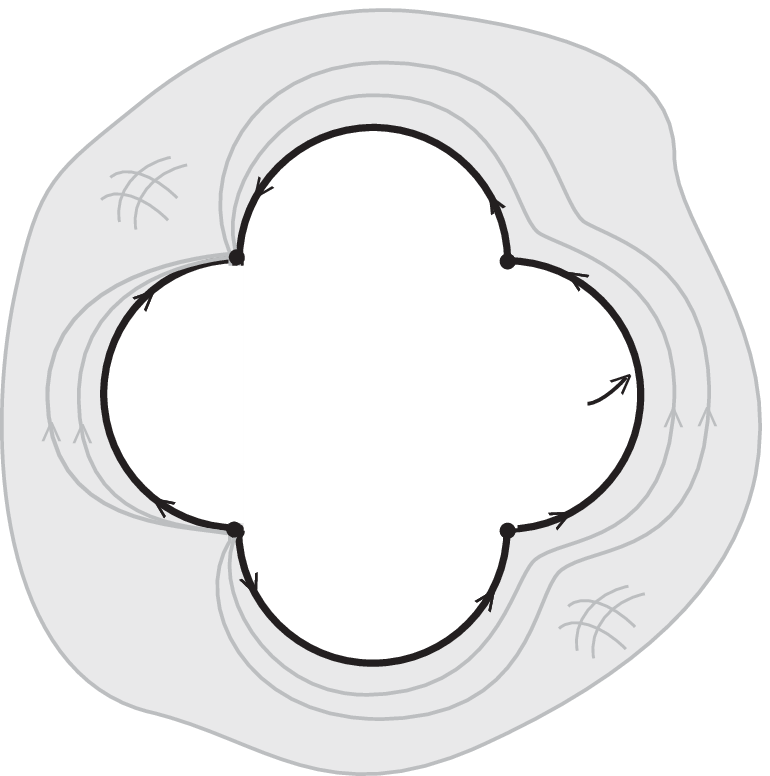}
\put (60,136) {Sink}
\put (130,136) {Saddle}
\put (130,79) {Saddle}
\put (55,79) {Source}
\put (86,105) {Node connection $\tau$}
\put (30,50) {$L'$}
\put (105,37) {$D'$}
\end{overpic}
\end{center}
\medskip

By  the nature of real blow-ups, the cardinal of $J$ is even and, if $J\neq\emptyset$, the number of connected components of $D'\setminus J$ and the number of node connections are also even.

To conclude the proposition, let us prove that  in this situation we have   $I_L(X)>0$.
The index $I_L(X)$ can be calculated as follows. Let $\mathcal{S}$ be a closed simple curve in $L'\setminus D'$ surrounding $D'$ and homotopic to $C'$ in $L'\setminus D'$
with the standard orientation and let $\phi:\SSS^1\to\mathcal{S}$ be an orientation preserving homeomorphism. Then $I_L(X)$ is equal to the degree of the map
$$
\theta:\SSS^1\to\SSS^1,\;p\mapsto\frac{X'(\phi(p))}{\|X'(\phi(p))\|}.
$$
Suppose, moreover, that $\mathcal{S}$ is a differentiable curve having only finitely many tangencies with $X'$ and let $i$ and $e$ be, respectively, the number of {\em interior} and {\em exterior} tangencies (i.e.,  at such a  tangency  point $q$, the orbit of $X$, devoid of $q$, stays locally at $q$ in the interior or  in  the exterior of $\mathcal{S}$, respectively). Then,  from Poincar\'{e} (see also Pugh's work \cite{Pug}),  we can calculate the degree of the map $\theta$ above as
\begin{equation}\label{eq:tangencies}
\mbox{deg}(\theta)=1+\frac{i-e}{2}.
\end{equation}
We will finish  by constructing  a differentiable curve  $\mathcal{S}$    with a positive (even) number of interior tangencies  and  no exterior tangencies with the vector field $X'$.

Let $\tau$ be a node connection in $D'$ and let $\g_1,\g_2$ be two trajectories of $X'$ in $L'\setminus D'$, both having $\alpha$ and $\omega$-limit equal to $\alpha(\tau)$ and $\omega(\tau)$, respectively, and such that $\g_2$ is inside the circle $\tau\cup\overline{\g_1}$. Using the flow-box theorem, we can construct a differentiable arc of curve $\mathcal{S}_\tau$ connecting two different points  of $\g_1$, lying inside $\tau\cup\overline{\g_1}$ except for its extremities and   everywhere transversal to $X'$ except for a point $a_\tau$ where it touches $\g_2$. This is depicted   in the    figure:


\medskip
\begin{center}
\begin{overpic}[scale=1]
{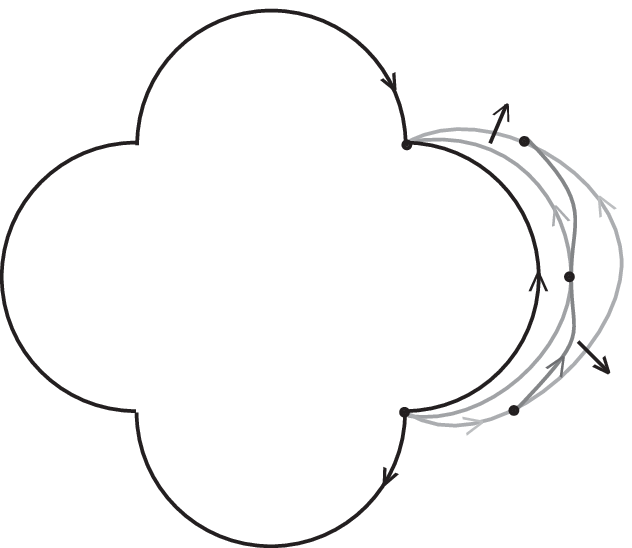}
\put (178,100) {$\gamma_{1}$}
\put (142,133) {$\gamma_{2}$}
\put (168,77) {$a_{\tau}$}
\put (177,46) {$\mathcal{S}_\tau$}
\put (146,77) {$\tau$}
\end{overpic}
\end{center}
\medskip

  We  choose the extremities of $\mathcal{S}_\tau$ sufficiently near the corresponding singular points $\alpha(\tau)$, $\omega(\tau)$, so that, in sufficiently small neighborhoods of the
 the node singularities of $D'$,
  the arcs $\mathcal{S}_\tau$ can be jointed in a smooth way by
 small arcs   transversal to $X'$. Thus, we produce   a simple closed curve $\mathcal{S}$ surrounding $D'$ with the required properties: $\mathcal{S}$ is everywhere transversal to $X'$ except for the points $a_\tau$, which are in fact interior tangencies of $\mathcal{S}$ with the vector field $X'$, and there are as  many  of them as the number of node connections (an even number).
\end{proof}

\begin{rk}{\em
It is worth  noticing  that   formula (\ref{eq:tangencies}) is closely related to Bendixson's formula for the index of a planar analytic vector field $X$ at the origin in $\R^2$:
$$
I(X)=1+\frac{e-h}{2}
$$
where $e$ is the number of ``elliptic'' sectors and $h$ is the number of ``hyperbolic'' sectors of $X'$ at the origin (see Andronov et al. \cite{And}). In fact, in our situation, if we collapse $L'$ into a neighborhood of $0\in\R^2$ sending $D'$ to the origin, the push-forward of $X'$ gives a vector field (which can be continuously extended to the origin) having as  many  elliptic sectors as the number of node connections in $D'$ and no hyperbolic sectors. This is an alternative proof of the last part of Proposition~\ref{pro:index-separatrix}, after the observation that Bendixson's formula extends to continuous vector fields which have  finitely many sectors of  elliptic, hyperbolic or parabolic   type.
}
\end{rk}

\section{Examples}\label{sec:example}

In this section we prove the second part of Theorem~\ref{th:main}, that is, we provide examples of analytic vector fields at $0\in\R^3$ having an analytic non-constant first integral but  not having  analytic separatrices.
Our examples are obtained as a one-parameter unfolding of a two-dimensional vector field which has a unique formal real separatrix which is not convergent. In the introduction, we have already discussed the existence of planar vector fields with such a property,  for instance, Risler's example in \cite{Ris}. We need to modify such  example in order that its unfolding produces a three-dimensional vector field with isolated singularity.

\begin{proposition}\label{pro:example-Ya}
Let $a=a(x)\in\R\{x\}$ be a convergent series in one variable such that $a(0)=a'(0)=0$ and consider the planar analytic vector field
\begin{equation}\label{eq:example-Ya}
Y_a=(y^2+x^4)\frac{\partial}{\partial x}+(-xy+x^3a(x)+\frac{a(x)}{x}y^2)\frac{\partial}{\partial y}.
\end{equation}
Then $Y_a$ has a unique real formal separatrix $\G_a$ at $0\in\R^2$ and, for a convenient (in fact generic) choice of the series $a(x)$, $\G_a$ is not convergent.
\end{proposition}

Using this proposition, we construct our desired examples in $\R^3$ as follows.

\begin{example}\label{ex:counter-example}{\em
Given $a(x)\in\R\{x\}$ with $a(0)=a'(0)=0$, consider the vector field in $\R^3$, expressed in coordinates $(x,y,z)$ as
$$
X_a=Y_a+z^2\frac{\partial}{\partial x},
$$
where $Y_a$ is given in (\ref{eq:example-Ya}). The vector field $X_a$ is in fact a family of planar vector fields in the parameter $z$. In other words, the function $f=z$ is an analytic first integral of $X_a$. Moreover, since the coefficient of $\partial/\partial x$ is $y^2+x^4+z^2$, the origin is an isolated singularity of $X_a$. Hence, the real formal separatrices of $X_a$ are those contained in the fiber $z=0$. More specifically, they are the separatrices of the restriction $X_a|_{z=0}=Y_a$. By Proposition~\ref{pro:example-Ya}, $X_a$ has a unique real formal separatrix $\G_a$, which is not convergent for a convenient choice of the series $a(x)$.
}
\end{example}

\strut

\noindent {\em Proof of Proposition~\ref{pro:example-Ya}.} If $\G$ is a formal real separatrix of $Y_a$ then its tangent line corresponds to a root  of the tangent cone of $Y_a$ at the origin, which is given by the equation $y^3+yx^2=y(y^2+x^2)=0$. Thus $\G$ is tangent to $\ell=(y=0)$. Let $\pi_1:M_1\to\R^2$ be the blow-up at the origin and let $p_1$ be the point in the exceptional divisor $E_1=\pi_1^{-1}(0)$ corresponding to $\ell$. The strict transform $\overline{\G}$ of $\G$ by $\pi_1$ is a formal separatrix of the the strict transform $\overline{Y}_a$ of $Y_a$ at $p_1$. A computation using usual coordinates $(x,y_1)=(x,y/x)$ of the blow-up $\pi_1$ shows that $\overline{Y}_a$ has a saddle-node singularity at $p_1$ for which the divisor $E_1$ is the {\em strong} separatrix (tangent to the non-zero eigenvalue) and thus $\overline{\G}$ is the {\em weak} formal separatrix (tangent to the zero eigenvalue). This proves the uniqueness of $\G=\G_a$.

Let us prove that $\G_a$ is not convergent for some choice of the series $a(x)$. For that, we consider the blow-up $\pi_2:M_2\to M_1$ at the point $p_1$ and the point $p_2$ in the exceptional divisor $E_2=\pi_2^{-1}(p_1)$ corresponding to the tangent of $\overline{\G}_a$ at $p_1$. We put usual coordinates at $p_2$ of the form $(x,y_2)=(x,y_1/x)=(x,y/x^2)$ and compute the strict transform of $\overline{Y}_a$ as
$$
\overline{\overline{Y}}_a=x^3(1+y_2^2)\frac{\partial}{\partial x}+
\left(-y_2(1+2x^2(1+y_2^2))+a(x)(1+y_2^2)\right)\frac{\partial}{\partial y_2}.
$$
Again $\overline{\overline{Y}}_a$ has a saddle-node singularity for which the divisor $E_2=(x=0)$ is the strong separatrix and the strict transform  $\overline{\overline{\G}}_a$ of $\overline{\G}_a$ by $\pi_2$ is the weak separatrix.  To finish, let us  show that $\overline{\overline{\G}}_a$ is not convergent for a convenient choice of $a$.
 Let as assume that $a(x)=\alpha(2x^2)$ for some $\alpha(z)\in z\R\{z\}$.  After dividing $\overline{\overline{Y}}_a$ by $1+y_2^2$, we consider the ramification $z=2x^2$ and   rename $w=y_2$,     obtaining the saddle-node vector field
\begin{equation}\label{eq:xia}
\xi_\alpha=z^2\frac{\partial}{\partial z}+\left(-w(1+z)+\frac{w^3}{1+w^2}+\alpha(z) \right)\frac{\partial}{\partial w}.
\end{equation}

It suffices to prove the following:

\begin{assertion}
  There is a choice of the series $\alpha(z)$ so that, for any $\delta>0$  sufficiently small, the weak formal separatrix of the saddle-node vector field $\xi_{\delta\alpha}$ is not convergent.
  \end{assertion}

We use the Martinet-Ramis moduli for analytic orbital classification of holomorphic foliations generated by saddle-node vector fields at the origin of $\C^2$ (see \cite{Mar-R}
and also \cite{Ily}). In our particular case, any vector field $\xi_\alpha$ of the form (\ref{eq:xia}) is formally orbitally equivalent to the   vector field in normal form
$$
N=z^2\frac{\partial}{\partial z}+(-w(1+z))\frac{\partial}{\partial w}.
$$
If we denote by $\mathcal{N}$ the class of vector fields formally orbitally equivalent to $N$, the moduli map associates to any $\eta\in\mathcal{N}$  is  a couple
$
G(\eta)=(g(\eta),\psi(\eta))
$
where $g(\eta)\in\C$ and $\psi(\eta)$ is a germ of a tangent to the identity biholomorphism at $(\C,0)$
in such a way that two vector fields $\eta,\eta'$ are orbitally analytically equivalent if and only if $G(\eta)=G(\eta')$. On the other hand, if $\eta\in\mathcal{N}$ then the weak formal separatrix of $\eta$ is convergent if and only if the constant part  $g(\eta)$  of the moduli is  equal to zero   \cite[Theorem III.4.4]{Mar-R}. Moreover, if we have a family $\{\eta_\lambda\}$ of vector fields in $\mathcal{N}$ depending analytically on $\lambda\in\C^m$ then $\lambda\mapsto g(\eta_\lambda)$ is also analytic \cite[Theorem 1, p. 33]{Ily}.

In order to prove the assertion, put $\delta=\eps^{3/2}$ for $\eps\in\R_{>0}$ and write the vector field $\xi_{\delta\alpha}$ under the change of variable $w=\sqrt{\eps}\bar{w}$ as
$$
\eta_{\eps,\alpha}=z^2\frac{\partial}{\partial z}+\left(-\bar{w}(1+z)+\eps \left( \bar{w}^3 +\alpha(z) \right)-\eps^2\bar{w}^5+\eps^3\bar{w}^7-\cdots \right)\frac{\partial}{\partial \bar{w}}.
$$
Hence $g(\xi_{\eps^{3/2}\alpha})= g(\eta_{\eps,\alpha})$ and it suffices show that there exists a series $\alpha=\alpha(z)$ so that
\begin{equation}\label{eq:derivative-moduli}
\frac{d(g(\eta_{\eps,\alpha}))}{d\eps}\mid_{\eps=0}\neq 0.
\end{equation}
(Notice that this gives the assertion since the weak separatrix of $\xi_{0}=\eta_{0,\alpha}$ is $z=0$ and hence $g(\xi_{0})=0$). First, put
$$
\overline{\eta}_{\eps,\alpha}=N+\eps(\bar{w}^3+\alpha(z))\frac{\partial}{\partial \bar{w}}=
z^2\frac{\partial}{\partial z}+\left(-\bar{w}(1+z)+\eps(\bar{w}^3+\alpha(z))\right)\frac{\partial}{\partial \bar{w}},
$$
so that (changing the notation $w=\bar{w}$) we have $\eta_{\eps,\alpha}=\overline{\eta}_{\eps,\alpha}+\eps Y_\eps$ where
$$
Y_\nu=(-\nu w^5+\nu^2 w^7-\cdots)\frac{\partial}{\partial w}.
$$
In other words, if we put $\zeta_{\eps,\nu,\alpha}=\overline{\eta}_{\eps,\alpha}+\eps Y_\nu$ then we have $\eta_{\eps,\alpha}=\zeta_{\eps,\eps,\alpha}$.  Notice that, for any series $\alpha$, we have that $g(\zeta_{\eps,\nu,\alpha})$ is analytic in $(\eps,\nu)$, $g(\zeta_{0,\nu,\alpha})=0$ for any $\nu$ and  $\zeta_{\eps,0,\alpha}=\overline{\eta}_{\eps,\alpha}$ for any $\eps$. Hence   we obtain
$$
\frac{d(g(\eta_{\eps,\alpha}))}{d\eps}\mid_{\eps=0}=
\frac{d(g(\overline{\eta}_{\eps,\alpha}))}{d\eps}\mid_{\eps=0}.
$$
Thus, to prove (\ref{eq:derivative-moduli}), it suffices to show that $\frac{d(g(\overline{\eta}_{\eps,\alpha}))}{d\eps}\mid_{\eps=0}\neq 0$ for some choice of $\alpha$.

The derivative of $g(\overline{\eta}_{\eps,\alpha})$ at $\eps=0$ (considered as a component of the tangent of the moduli map $G$) can be  computed explicitly from Elizarov's paper \cite{Eli} as follows.
Make the change of variables $z \mapsto -z$ and multiply by $-1$, getting the new expression for the family
$$
\overline{\eta}_{\eps,\alpha}=
z^2\frac{\partial}{\partial z}+ \left(w(1-z)-\eps(w^3 + \alpha(-z))\right)\frac{\partial}{\partial w}.
$$
To put it in Elizarov's pattern, we have to divide it by $1-z$
so that
the family $\bar{\eta}_{\eps,\alpha}$ becomes the family $v_{p,\lambda}+ \eps P\partial_w$ considered in equation \cite[Eq. 1.8]{Eli}, where
\[ v_{p,\lambda} = \frac{z^2}{1-z} \frac{\partial}{\partial z}+w  \frac{\partial}{\partial w}  \ \ \ \text{(and hence  $p=1$ and  $\lambda=-1$)} \]
 and
$$
  P(z,w) = -\frac{\alpha(-z)+w^3}{1-z} = -(\alpha(-z)+w^3)(1 + z + z^{2} + \cdots).
  $$
Choose $\alpha(z)$ such that $\alpha(0) = \alpha'(0) = 0$ and write
$$-\alpha(-z)(1 + z + z^{2} + \cdots) = \sum_{k \geq 2}c_k z^k.$$
This corresponds to $f_{-1}(z)$ in the expansion in power series in \cite[Eq. 1.9]{Eli}.
The constant part $g$ of the moduli map corresponds in our case to the component $a_{0,-1}$ in equation \cite[Eq. 1.3]{Eli}  (that is, $j=0$ and $l=-1$).

From all these data, and computing the sequence $m_k(l)=m_k(-1)$ in \cite[Eq. 1.7]{Eli} for the corresponding {\em Borel transform}, we conclude
  from Elizarov's formula in \cite[Theorem 1]{Eli} that
  $$
  \frac{d(g(\overline{\eta}_{\eps,\alpha}))}{d\eps}\mid_{\eps=0}=u \sum_{k=2}^{\infty}  c_{k}  \frac{k}{\Gamma(k+2)},
$$
where $\Gamma$ is the Euler's Gamma function and $u$ is some non-zero constant which does not depend on $\alpha$ (if we want to be precise, we can check that in fact $u=-1$).
 Therefore,   $\frac{d(g(\overline{\eta}_{\eps,\alpha}))}{d\eps}\mid_{\eps=0} \neq 0$ for   a generic choice of $\alpha(z)$,
 as we wanted. This ends the proof. \hfill{$\square$}



\end{document}